\documentclass{amsart}[12pt]
\usepackage{amssymb,amsmath}
\usepackage{enumerate}
\usepackage[english]{babel}
\usepackage{color}

\newcommand {\ii} {\infty}

\newcommand {\al} {\alpha}

\newcommand {\lb} {\lambda}

\newcommand {\sm} {\setminus}
\newcommand {\su} {\subset}

\newcommand {\mc} {\mathcal}

\newtheorem{teo}{Theorem}[section]
\newtheorem{pro}{Proposition}[section]

\theoremstyle{definition}
\newtheorem{rem}{Remark}[section]

\title{Skew-Hermitian operators in real Banach spaces of self-adjoint compact operators  }
\keywords{Symmetric sequence space, Banach ideal of compact operators, skew-Hermitian operator}
\subjclass[2010]{46L52, 47B10, 47C15}

\begin{document}
\date{July 12, 2019}

\begin{abstract}
Let $\mathcal H$ be a  complex infinite-dimensional separable Hilbert space, and let $\mathcal K(\mathcal H)$ be the $C^*$-algebra of  compact linear operators in $\mathcal=l H$. Let $(E,\|\cdot\|_E)$ be a  symmetric sequence space. If $\{\mu(n,x)\}$ are the singular values of $x\in\mathcal K(\mathcal H)$, let $\mc C_E=\{x\in\mathcal K(\mc H): \{\mu(n,x)\}\in E\}$ with $\|x\|_{\mathcal C_E}=\|\{\mu(n,x)\}\|_E$, $x\in\mathcal C_E$, be the Banach ideal of compact operators generated by $E$.  Let  $\mathcal C_E^h=\{x\in \mathcal C_E : x=x^*\}$ be the real Banach subspace of self-adjoint operators in $(\mathcal C_E, \|\cdot\|_{\mathcal C_E})$.  We show that in the case when $\mc C_E$ is a  separable or perfect Banach symmetric  ideal, $\mathcal C_E \neq \mathcal C_{l_2}$, for any skew-Hermitian operator $H\colon\mathcal C_E^h \to \mathcal C_E^h$ there exists self-adjoint bounded linear operator $a$ in $\mc H$ such that $H(x)=i(xa - ax)$ for  all $x\in\mathcal C_E^h$.
\end{abstract}

\author{B.R. Aminov}
\address{National University of Uzbekistan\\
Tashkent, 100174, Uzbekistan}
\email{aminovbehzod@gmail.com}
\author{V.I. Chilin}
\address{National University of Uzbekistan\\
Tashkent, 100174, Uzbekistan}
\email{vladimirchil@gmail.com; chilin@ucd.uz}

\maketitle

\section{Introduction}

Let $(\mathcal H, (\cdot, \cdot))$ be an infinite-dimensional complex separable Hilbert space, and  let $\mathcal B(\mathcal H)$ (respectively, $\mathcal K(\mathcal H)$) be the $C^*$-algebra of  all bounded (respectively, compact) linear operators on $\mathcal H$.
For a compact operator $x\in\mathcal K(\mathcal H)$, we denote by $ \big\{\mu(n,x)\big\}_{n = 1}^{\infty}$ the singular value sequence of $x$, \ that is, the decreasing rearrangement of the eigenvalue sequence of $| x|= (x^*x)^{\frac12}$. We let $\mathrm{Tr}$ denote the standard trace on $\mathcal B(\mathcal H)$. For $p\in[1,\infty)$ ($p=\infty$), we let
\[\mathcal C_p:=\Big\{x\in\mathcal K(\mathcal H)\,:\,\mathrm{Tr}\big(|x|^p\big)<\infty\Big\} \ \ (respectively, \mathcal C_\infty=\mathcal K(\mathcal H) ) \]
denote the $p$-th Schatten ideal of $\mathcal B(\mathcal H)$, with the norm
\[\|x\|_p:= \mathrm{Tr}\big(|x|^p\big)^{\frac1p},\quad (respectively, \ \|x\|_\infty:= \sup\limits_{n\geq1}|\mu(n,x)|).\]

The Schatten ideals $\mathcal C_p= \mathcal C_{l_p}$ are examples of Banach symmetric ideals $\mathcal C_E=\{x\in\mc K(\mc H):  \big\{\mu(n,x)\big\}_{n = 1}^{\infty}\in E\}$ with norm $\|x\|_{\mc C_E}=\| \big\{\mu(n,x)\big\}_{n = 1}^{\infty}\|_E$ of compact operators generated by symmetric sequence spaces  $(E, \|\cdot\|_{E})$ (see section 2 below).

Let $[\cdot,\cdot]$ be a semi-inner product on $\mathcal C_E$ compatible with the norm  $\|\cdot\|_{\mc C_E}$, that is, $\|x\|_{\mc C_E} = \sqrt{[x,x]}$ for all $x \in \mc C_E$ \ \cite[Ch. 2, \S 1]{dragomir}. A bounded linear operator $H: \mc C_E \to \mc C_E$  is called {\it Hermitian} if  $[Hx, x]$ is real for all $x \in  \mc C_E$ \cite[Ch. 5, \S 2]{flem1}.

In 1981  A.Sourour \cite{sourour} gave the following description of all the Hermitian operators acting in separable Banach symmetric ideal.
\begin{teo}\label{t11}
Let  $(\mathcal C_E, \|\cdot\|_{\mathcal C_E})$ be a separable Banach symmetric ideal, and let  $\mathcal C_E \neq \mathcal C_2$. Then for any Hermitian operator $H\colon \mathcal C_E \to \mathcal C_E$ \  there are self-adjoint operators $a, b \in \mathcal B(\mathcal H)$ such that $H(x) = ax + xb$ for all $x \in \mc C_E$.
\end{teo}
In \cite{ach}, a variant of Theorem \ref{t11} was obtained  for any perfect Banach symmetric ideals  $(\mathcal C_E, \|\cdot\|_{\mathcal C_E}), \ \mathcal C_E \neq \mathcal C_2$ \ ( recall that $(\mathcal C_E, \|\cdot\|_{\mathcal C_E})$ is a perfect ideals, if  $\mathcal C_E = \mathcal C_{E}^{\times\times}$ \cite{garling} (see section 2 below)).

Let $\mathcal C_E^h=\{x\in \mathcal C_E : x=x^*\}$ be a Banach real subspace  in Banach symmetric ideals  $(\mathcal C_E, \|\cdot\|_{\mathcal C_E})$. A linear bounded operator $H\colon \mathcal C_E^h \to \mathcal C_E^h$ is said to be {\it skew-Hermitian}, if $[H(x),x]=0$ for all $x\in \mathcal C_E^h$, where $[\cdot,\cdot]$ is a semi-inner product on $\mathcal C_E$ compatible with the norm  $\|\cdot\|_{\mc C_E}$.

It is clear that the linear operator $H\colon \mathcal C_E^h \to \mathcal C_E^h$ defined by $H(x)=i(xa - ax)$, where $a=a^*\in  \mathcal B(\mathcal H)$, \ $i^2 = -1$, \ is a skew-Hermitian operator.

Our main result states that if $(\mathcal C_E, \|\cdot\|_{\mathcal C_E})$ is a separable or a perfect Banach symmetric ideal of compact operators, $\mathcal C_E \neq \mathcal C_2$, then there are no other skew-Hermitian operators in $(\mathcal C_E^h, \|\cdot\|_{\mathcal C_E})$:
\begin{teo}\label{t12}
Let $(\mathcal C_E, \|\cdot\|_{\mathcal C_E})$ be a separable or a perfect Banach symmetric ideal, \  $\mathcal C_E \neq \mathcal C_2$, and let  $H\colon \mathcal C_E^h \to \mathcal C_E^h$  be a skew-Hermitian operator. Then there exists self-adjoint operator $a\in \mathcal B(\mathcal H)$ such that $H(x)=i(xa - ax)$ for all $x \in \mc C_E^h$.
\end{teo}

\section{Preliminaries}

Let $\ell_{\infty}$ (respectively,  $c_0$)  be the Banach space of bounded (respectively, converging to zero) sequences $\{\xi_n\}_{n=1}^\ii$ of complex numbers equipped with the norm $\|\{\xi_n\}\|_\ii=\sup\limits_{n \in \mathbb N} |\xi_n|$, where $\mathbb N$ is the set of natural numbers. If $2^{\mathbb N}$ is the $\sigma$-algebra of subsets of $\mathbb N $ and
$\mu(\{n\})=1$ for each $n\in\mathbb N$, then $(\mathbb N, 2^{\mathbb N},\mu)$ is a $\sigma$-finite measure space such that $\mathcal L_\ii(\mathbb N, 2^{\mathbb N},\mu)=\ell_{\infty}$ and
\[
\mathcal L_1(\mathbb N, 2^{\mathbb N},\mu)=\ell_1=\left \{\{ \xi_n \}_{n = 1}^\ii\su\mathbb C:\ \|\{ \xi_n \} \|_1
=\sum_{n=1}^\ii |\xi_n|<\ii\right \}\subset \ell_{\infty},
\]
where $\mathbb C$ is the field of complex numbers.

If $\xi=\{\xi_n\}_{n = 1}^\ii\in \ell_{\infty}$, then the {\it non-increasing rearrangement} $\xi^*:(0,\ii) \to (0,\ii)$ of $\xi$ is defined by
\[
\xi^*(t)=\inf\{\lb:\mu\{|\xi|>\lb\}\leq t\},\ \ t > 0,
\]
(see, for example, \cite[Ch.\,2, Definition 1.5]{bs}). As such, the non-increasing rearrangement of
a sequence $\{\xi_n\}_{n = 1}^\ii \in \ell_{\infty}$ can be identified with the sequence $\xi^*=\{\xi_n^*\}_{n=1}^\ii$, where
\[
\xi_n^*= \inf\left\{ \sup\limits_{n \notin F} |\xi_n |: F\su \mathbb N,\ |F|<n\right\}.
\]
If $\{\xi_n\}\in c_0$, then $\xi_n^* \downarrow 0$; in this case there exists a bijection $\pi:\mathbb N \rightarrow \mathbb N$ such that $|\xi_{\pi(n)}| = \xi_n^*$, $n\in\mathbb N$.

{\it Hardy-Littlewood-Polya partial order} in the space  $\ell_{\infty}$ is defined as follows:
\[
\xi=\{\xi_n\}\prec\prec\eta=\{\eta_n\}\ \ \Longleftrightarrow \ \ \sum_{n=1}^m \xi^*_n \leq \sum_{n=1}^m \eta^*_n \ \text{\ \ for all\ \ } \ m\in\mathbb N.
\]

A non-zero linear subspace $E\su \ell_{\infty}$ with a Banach norm $\|\cdot\|_E$ is called a {\it symmetric} ({\it fully symmetric}) sequence space if
\[
\eta \in E, \  \xi \in \ell_{\infty}, \  \xi^* \leq \eta^*\text{\ (resp.},\ \xi^*\prec \prec \eta^*) \ \Longrightarrow \ \xi \in E \text{\ \ and\ \ }\| \xi \|_E \leq \| \eta\|_E.
\]
Every fully symmetric sequence space is a symmetric sequence space. The converse is not true in general. At the same time, any separable symmetric sequence space  is a fully symmetric space.

If $(E,\|\cdot\|_E)$ is a symmetric sequence  space  and $E_h= \{\xi=\{\xi_n\}_{n = 1}^\ii  \in E : \xi_n \in \mathbb R \ \ \forall  \ \ n \in \mathbb N \}$, where  $\mathbb R$ is the field of real numbers, then $(E_h,\|\cdot\|_E)$ is a Banach lattice with respect to the natural partial order
\[
\{\xi_n\}_{n = 1}^\ii=\xi \leq\eta=\{\eta_n\}_{n = 1}^\ii \ \Longleftrightarrow \ \xi_n \leq \eta_n \ \ \text{for all} \ \ \ n \in \mathbb N,
\]
and, in addition,
\[
\|\xi\|_E=\|\,|\xi|\,\|_E = \|\xi^*\|_E \text{ \ \ for all\ \ }\xi\in E_h.
\]

Examples of fully symmetric sequence spaces are  $(\ell_{\infty},\|\cdot\|_\ii)$, \  $(c_0,\|\cdot\|_\ii)$ and
\[
\ell_p=\left\{\xi=\{\xi_n\}_{n=1}^\ii\in c_0:\ \|\xi\|_p=\left(\sum_{n=1}^\ii|\xi_n|^p\right)^{1/p}<\ii\right\},\ 1\leq p<\ii.
\]
For any symmetric sequence space $(E,\|\cdot\|_{E})$ the following continuous embeddings hold \cite[Ch.\,2, \S\,6, Theorem 6.6]{bs}:
$$
(\ell_1,\|\cdot\|_1)\su (E,\|\cdot\|_E) \subset (\ell_{\infty},\|\cdot\|_\ii).
$$
Besides, $\|\xi\|_E\leq \|\xi\|_1$ for all $\xi \in l_1$ and  $\|\xi\|_\ii\leq \|\xi\|_E$ for all $\xi \in E$.

If there is $\xi\in E\sm c_0$, then $\xi^*\geq\al\mathbf1$ for some $\al>0$, where $\mathbf1 = \{1,1,...\}$. Consequently, $\mathbf1\in E$ and $E =\ell_{\ii}$. Therefore, either $E\su c_0$ or $E=\ell_\ii$.

Now, let $(\mc H,(\cdot, \cdot))$ be a  complex infinite-dimensional separable Hilbert space, and let $(\mc B(\mc H),\|\cdot\|_\ii)$ be the $C^*$-algebra of bounded linear operators in $\mc H$.  Denote by $\mc K(\mc H)$
(respectively, $\mc F(\mc H)$) the two-sided ideal of compact (respectively, finite rank) linear operators in $\mc B(\mc H)$. It is well known that, for any proper two-sided ideal $\mc I\su\mc B(\mc H)$, we have $\mc F(\mc H)\su\mc I \su\mc K(\mc H)$ (see, for example, \cite[Proposition 2.1]{simon}).

Denote $\mc B_h(\mc H)=\{x \in\mc B(\mc H): x =x^{\ast}\}$, $\mc B_+(\mc H)=\{x \in\mc B(\mc H): x\ge 0\}$, and let
$\mathrm{Tr}:\mc B_+(\mc H)\to [0, \ii]$ be the {\it canonical trace} on $\mc B(\mc H)$, that is,
\[
\mathrm{Tr}(x)=\sum\limits_{j \in J} (x\varphi_j,\varphi_j), \ \ x\in\mc B_+(\mc H),
\]
where $\{\varphi_j\}_{j \in J}$ is an orthonormal basis in $\mc H$ (see, for example, \cite[Ch.\,7, E.\,7.5]{sz}).

Let $\mc P(\mc H)$ be the lattice of projections in $\mc H$. If $\mathbf 1$ is the identity of $\mc B(\mc H)$ and
$e\in\mc P(\mc H)$, we will write $e^\perp=\mathbf 1-e$.

Let $x\in\mc B(\mc H)$, and let $\{e_\lb\}_{\lb\ge 0}$ be the spectral family of projections for the absolute value
$|x|= (x^*x)^{1/2}$ of $x$, that is, $e_\lb=\{|x|\leq\lb\}$.
If $t>0$, then the {\it $t$-th generalized singular number} of $x$, or the {\it non-increasing rearrangement} of $x$,
is defined as
\[
\mu_t(x)=\inf\{\lb>0:\ \mathrm{Tr}(e_\lb^\perp)\leq t\}
\]
(see \cite{fk}).

A non-zero linear subspace  $X\su\mc B(\mc H)$ with a Banach norm $\|\cdot\|_X$ is called {\it symmetric} ({\it fully symmetric}) if the conditions
\[
x\in X, \ y\in\mc B(\mc H), \ \mu_t(y)\leq \mu_t(x)\text{ \ \ for all\ \ } t>0 \
\]
(respectively,
\[
x\in X, \ y\in\mc B(\mc H), \ \int\limits_0^s\mu_t(y)dt\leq\int\limits_0^s\mu_t(x)dt \text{ \ \ for all\ \ } s>0 \ \  (\text {writing} \  y \prec\prec x))
\]
imply that $y\in X$ and $\| y\|_X\leq \| x\|_X$.

The spaces $(\mc B(\mc H),\|\cdot\|_\ii)$ and $(\mc K(\mc H),\|\cdot\|_\ii)$ as well as the classical Banach two-sided ideals
\[
\mc C_p =\{x\in\mc K(\mc H) :\ \|x\|_p=\mathrm{Tr}(|x|^p)^{1/p}<\ii\}, \  1 \leq p<\ii,
\]
are examples of fully symmetric spaces.

It should be noted  that  for every symmetric space  $(X,\|\cdot\|_X) \su \mc B(\mc H)$ and all $x\in X$, $a, b\in\mc B(\mc H)$,
\[
\|x\|_X = \|\,|x|\,\|_X = \|x^*\|_X, \ \ axb \in X,\text{\ \ and\ \ }\|axb\|_X \leq \|a\|_\ii\|b\|_\ii\|x\|_X.
\]

\begin{rem}\label{r1}
If $X\su\mc B(\mc H)$ is a symmetric subspace and there exists a projection $e\in \mc P(\mc H)\cap X$ such that $\mathrm{Tr}(e)=\ii$, that is, $\dim e(\mc H)= \ii$, then $\mu_t(e)=\mu_t(\mathbf 1)= 1$ for every $t \in (0, \infty)$. Consequently,  $\mathbf 1 \in X$ and $X = \mc B(\mc H)$. If $X \neq \mc B(\mc H)$ and $x \in X$, then $e_\lb=\{|x|>\lb\}$ is a finite-dimensional projection, that is,
$\dim e_\lb (\mc H)<\ii$ for all $\lb>0$. This means that $x\in\mc K(\mc H)$, hence $X\su\mc K(\mc H)$. Therefore, either $X=\mc B(\mc H)$ or $X\su\mc K(\mc H)$.
\end{rem}

If $x\in\mc K(\mc H)$, then $|x|=\sum\limits_{n=1}^{m(x)} \mu(n,x) p_n$ (if $m(x)=\ii$, the series converges uniformly),
where $\{\mu(n,x)\}_{n=1}^{m(x)}$ is the set of singular values of $x$, that is, the set of eigenvalues of the compact operator $|x|$ in the decreasing order, and $p_n$ is the projection onto the eigenspace corresponding to $\mu(n,x)$. Consequently, the non-increasing rearrangement $\mu_t(x)$ of $x\in\mc K(\mc H)$ can be identified with the sequence $\{\mu(n,x)\}_{n=1}^\ii$, $\mu(n,x) \downarrow 0$ (if $m(x)<\ii$, we set $\mu(n,x)=0$ for all $n>m(x)$).

Let $(X,\|\cdot\|_X)\su \mc K(\mc H)$ be a symmetric space. Fix an orthonormal basis  $\{\varphi_n\}_{n=1}^\infty$ in $\mc H$, and denote by $p_n$ be the  projection on the  one-dimension linear subspace $\mathbb C\cdot\varphi_{n}\su\mc H$.  It is clear that the set
\[
E(X)=\left\{\xi=\{\xi_n\}_{n = 1}^\ii\in c_0: \ x_\xi =\sum\limits_{n=1}^\ii\xi_np_n \in X\right\} \]
(the series converges uniformly),
is a symmetric sequence space with respect to the norm $\|\xi\|_{E(X)} = \| x_\xi\|_X$. Consequently, each symmetric subspace $(X,\|\cdot\|_X)\su\mc K(\mc H)$ uniquely generates a symmetric sequence space $(E(X), \|\cdot\|_{E(X)})\su c_0$. The converse is also true: every symmetric sequence space $(E,\|\cdot\|_E)\su c_0$ uniquely generates a symmetric space $(\mc C_E,\|\cdot\|_{\mc C_E})  \su \mc K(\mc H)$ by the following rule (see, for example, \cite[Ch.\,3, Section 3.5]{lsz}):
\[
\mc C_E =\{x\in\mc K(\mc H):\, \{\mu(n,x)\}\in E\},\ \ \|x\|_{\mc C_E} = \|\{\mu(n,x)\}\|_{E}.
\]
In addition,
\[
E(\mc C_E)=E, \ \|\cdot\|_{E(\mc C_E)}=\|\cdot\|_E, \ \mc C_{E(\mc C_E)}=\mc C_E, \ \|\cdot\|_{\mc C_{E(\mc C_E)} }=\|\cdot\|_{\mc C_E}.
\]

We will call the pair $(\mc C_E,\|\cdot\|_{\mc C_E})$ a {\it Banach ideal of compact operators} (cf. \cite[Ch. III]{gohb}, \cite[Ch. 1, \S 1.7]{simon}). It is known that
$(\mc C_p,\|\cdot\|_p)=(\mc C_{l_p},\|\cdot\|_{\mc C_{l_p}})$ for all $1\leq p<\ii$ and
$(\mc K(\mc H),\|\cdot\|_\ii)= (\mc C_{c_0},\|\cdot\|_{\mc C_{c_0}})$. In addition,   $\mathcal C_1 \subset \mathcal C_E \subset \mathcal K(\mathcal H)$ \ and \ $\|x\|_{\mathcal C_E} \leq  \|x\|_1, \ \ \|y\|_{\infty}\leq  \|y\|_{\mathcal C_E}$ for all $x \in \mathcal C_1, \ y \in \mathcal C_E.$ Note also that every separable Banach ideal of compact operators is a fully symmetric ideal.

If $(E, \|\cdot\|_{E})$ is a symmetric sequence space  (respectively, $(\mathcal C_E, \|\cdot\|_{\mathcal C_E})$ is  a Banach symmetric ideal), then the K\"{o}the  dual $E^\times$ (respectively, $\mathcal  C_E^\times$) is defined as
\[ E^\times=\{\xi=\{\xi_n\}_{n=1}^{\infty} \in \ell_{\infty}\,: \ \, \xi\eta=\{\xi_n\eta_n\}_{n=1}^{\infty}\in \ell_1 \ \ \text{for all} \ \ \eta =\{\eta_n\}_{n=1}^{\infty} \in E\},\]
\[ (\text{respectively}, \ \ \mathcal  C_E^\times=\{x\in \mathcal  B(\mathcal H)\,: \ \, xy\in \mathcal  C_1 \ \ \text{for all} \ \ y\in \mathcal  C_E\}),\]
and
\[ \|\xi\|_{E^\times}=\sup\limits\{\sum\limits_{n=1}^{\infty}|\xi_n \eta_n|: \eta=\{\eta_n\}_{n=1}^{\infty} \in E, \ \|\eta\|_{E}\leq 1  \}, \ \xi \in E^\times, \]
\[ (\text{respectively}, \ \ \|x\|_{\mathcal C_E^\times}=\sup\limits\{\mathrm{Tr}\big(|xy|\big): y\in \mathcal C_E, \ \|y\|_{ \mathcal C_E}\leq 1  \}, \  x \in \mathcal  C_E^\times). \]

It is known that $(E^\times, \|\cdot\|_{E^\times})$ is a  symmetric sequence space  \cite[Ch. II, \S 4, Theorems 4.3, 4.9]{kps} \ and  $\ell_1^\times = \ell_{\infty}$. In addition, if $E\neq \ell_1$ then $E^\times \subset c_0$. Therefore, if $E\neq \ell_1$, the space $(\mathcal  C_E^\times, \|\cdot\|_{\mathcal  C_E^\times})$ is a   symmetric ideal of compact operators.

A  Banach symmetric ideal  $(\mathcal  C_E, \|\cdot\|_{\mathcal C_E})$  is said to be {\it perfect} if $\mathcal C_E = \mathcal C_{E}^{\times\times}$ (see, for example, \cite{garling}). It
 is clear that $\mathcal C_E$  is perfect if and only if $E = E^{\times\times}$.

A symmetric  sequence space $(E, \|\cdot\|_{E})$
 (a   Banach symmetric ideal $(\mathcal  C_E, \|\cdot\|_{\mathcal C_E})$)  is said to possess
{\it Fatou property} if the conditions
\[
 0 \leq \xi_{k }\leq \xi_{k+1}, \  \xi_{k}\in E \ \ \text{(respectively}, \ 0 \leq \ x_{k }\leq x_{k+1}, \   x_{k}\in \mathcal C_E) \ \ \text{for all} \ \ k \in \mathbb N
\]
and \ $\sup\limits_{k\geq 1} \| \xi_{k}\|_E<\infty$ \ (respectively,  $\sup\limits_{k\geq 1} \| x_{k}\|_{\mathcal{C}_E}<\infty)$ imply that there exists an element $\xi \in E$ \ (respectively, $x \in \mathcal{C}_E)$ such that $\xi_{k}\uparrow \xi$ \ and  \ $\| \xi\|_E=\sup \limits_{k\geq 1} \| \xi_{k}\|_E$ \ (respectively, \ $x_{k}\uparrow x$ \ and \  $\| x\|_{\mathcal C_E}=\sup\limits_{k\geq 1} \| x_{k}\|_{\mathcal C_E})$.

It is known that $(E, \|\cdot\|_{E})$ (respectively, $(\mathcal C_E, \|\cdot\|_{\mathcal C_E})$   has the Fatou property if and only if $E = E^{\times\times}$ \cite[Vol. II, Ch. 1, Section a]{lt} (respectively,  $\mathcal C_E= \mathcal C_E^{\times\times}$ \cite[Theorem 5.14]{ddp}). Therefore $(\mathcal C_E, \|\cdot\|_{\mathcal C_E})$ is a perfect Banach symmetric ideal if and only if
$(\mathcal C_E, \|\cdot\|_{\mathcal C_E})$  has the Fatou property. Note  that every perfect Banach symmetric ideal is a fully symmetric  ideal.

If $y\in \mathcal C_E^{\times}$, then a linear functional $f_y(x)= \mathrm{Tr}\big(x\cdot y\big), \ x \in  \mathcal C_E$,  is continuous on  $(\mathcal C_E, \|\cdot\|_{\mathcal C_E})$, in addition, $\|f_y\|_{\mathcal C_E^{\ast}} = \|y\|_{\mathcal C_E^{\times}}$,   where $(\mathcal C_E^{\ast},\|\cdot\|_{\mathcal C_E^{\ast}})$ is the dual of the Banach space   $(\mathcal C_E, \|\cdot\|_{\mathcal C_E})$ (see, for example, \cite{garling}).
Identifying an element $y \in  \mathcal C_E^{\times}$ and the linear functional $f_y$, we may assume that $\mathcal C_E^{\times}$ is a closed linear subspace in $\mathcal C_E^{\ast}$.
Since $\mathcal F(\mathcal H) \subset \mathcal C_E^{\times}$, it follows that $\mathcal C_E^{\times}$ is a total subspace in $\mathcal C_E^{\ast}$, that is, the conditions $x \in \mathcal C_E, \ f(x)=0$ for all $f \in \mathcal C_E^{\times}$ imply  $x=0$. Thus, the weak topology $\sigma(\mathcal C_E, \mathcal C_E^{\times})$ is a Hausdorff topology, in addition   $\mathcal F(\mathcal H)$ (respectively, $\mathcal F(\mathcal H)^h$) is $\sigma(\mathcal C_E, \mathcal C_E^{\times})$-dense in $\mathcal C_E$ (respectively, in $\mathcal C_E^h$).

\section{Skew-Hermitian operators in $\mathcal C_E^h$}

{\it A semi-inner product} on a real linear space $X$ is a  form $[\cdot,\cdot]\colon X\times X\to \mathbb R$ which satisfies

$(i)$. $[\alpha x+ y,z] = \alpha\cdot[x,z]+[y,z]$ for all $\alpha \in \mathbb R$ \ and \  $ x, y, z \in X$;

$(ii)$. $[x, \alpha y] = \alpha\cdot[x,y]$ for all $\alpha \in \mathbb R$ \ and \  $ x, y \in X$;

$(iii)$. $[x,x] \geq 0$ for all $ x \in X$ \ and \ $[x,x] =0$ implies that $x = 0$;

$(iv)$. $| [x,y] |^2 \leq [x,x]\cdot [y,y]$ for all $ x, y \in X$. \\
(see, for example, \cite[Ch. 2, \S 1]{dragomir}).

The function $\|x\| = \sqrt{[x,x]}$ is the norm  on a linear space $X$. Conversely, if $(X,\|\cdot\|_X)$ is a  normed real linear space, then there exists  semi-inner product $[\cdot,\cdot]$ on $X$ {\it compatible with the norm  $\|\cdot\|_X$}, that is, $\|x\|_X = \sqrt{[x,x]}$ for all $ x \in X$ \cite[Ch. 2, \S 1]{dragomir}. In particular, the semi-inner product, which is compatible with the norm  $\|\cdot\|_X$,  can be defined using the equation $[x,y]= \varphi_y (x)$, where $\varphi_y\in X^*$, $\|\varphi_y\|_{X^*} = \|y\|_X$ and $\varphi_y (y) = \|y\|_X^2$ (such functional  is called a \emph{support functional} at $y \in X$) (\cite[Ch. 2, \S 1, Theorem 10]{dragomir}.

Let $(X,\|\cdot\|_X)$ be a real Banach  space, and let  $[\cdot,\cdot]$ be a semi-inner product on  $X$ which is compatible with the norm  $\|\cdot\|_X$. A linear bounded operator $H\colon X\to X$ is said to be {\it skew-Hermitian}, if $[H(x),x]=0$ for all $x\in X$ (\cite{flem2}, Ch. 9, \S 4), in particular, $\varphi_{x} (H(x))=0$ for every $x\in X$.

The following Proposition is well known (\cite[ Ch. 9, \S 4, Proposition 9.4.2]{flem2}).

\begin{pro}\label{p31}
Let $(X,\|\cdot\|_X)$ be a real Banach space and let $H\colon X\to X$ be a skew-Hermitian operator. If \ $V\colon X\to X$ is a surjective linear isometry then an  operator  $V\cdot H\cdot V^{-1}$ is a skew-Hermitian.
\end{pro}

Let $(\mathcal C_E, \|\cdot\|_{\mathcal C_E})$ be a separable or perfect Banach symmetric  ideal, \  $\mathcal C_E \neq \mathcal C_2$. Let $H\colon \mathcal C_E^h\to \mathcal C_E^h$  be a  skew-Hermitian operator. We want to prove Theorem \ref{t12}, i.e. we will show that there exists $a\in  \mathcal B(H)^h$ such that $H(x)=i(xa - ax)$ for  all $x\in \mathcal C_E^h$. To solve this problem, we use a modification of the the original proof of Sourour Theorem 1 \cite{sourour}.

For vectors $\xi, \eta \in \mathcal H$, denote by  $\xi\otimes\eta $ the rank one operator on $\mathcal H$ defined by the equality $ (\xi\otimes\eta)(h)=(h,\eta)\xi, \ h \in \mathcal H$. It is easily seen $$\langle x, \xi\otimes \eta \rangle: = \mathrm{Tr}((\eta\otimes \xi)\cdot x)=(x(\eta),\xi)$$ for any $x\in \mathcal B(\mathcal H)^h$ and $\xi, \eta \in \mathcal H$.
If $y=\xi\otimes \xi, \ \|\xi\|_{\mathcal H} = 1$, then  $y$ is an one dimensional projection on $\mathcal H$ \ and \ $\|y\|_{\mathcal  C_E} = \|y\|_\infty =1$. Thus for a linear functional $$f_y(x):= \langle x, y \rangle= \mathrm{Tr}(y^*x), \ x\in \mathcal C_E^h,$$  we have that \
$$f_y(y)=\mathrm{Tr}(y^2)=\mathrm{Tr}(y)=(\xi,\xi)=1=\|y\|_{\mathcal C_E}^2.$$ \
In addition, if $x\in \mathcal C_E^h$ and $\|x\|_{\mathcal C_E}\leq 1$ then \
$$|f_y(x)| = |\mathrm{Tr}(yx)|=|(x(\xi),\xi)|\leq  \|x\|_{\infty}\leq \|x\|_{\mathcal C_E} \leq 1.$$
Consequently, $\|f_y\|_{(\mathcal C_E^h)^*} =1 = \|y\|_{\mathcal C_E}$. This means that   $f_y$ is a support functional at $y \in \mathcal C_E^h$, and $[x,y]= f_y(x)$ is a semi-inner product  on $\mathcal C_E^h$ compatible with the norm  $\|\cdot\|_{\mathcal C_E^h}$ \ (\cite[Ch. 2, \S 1, Theorem 10]{dragomir}.

\begin{pro}\label{p32} If $\xi, \eta\in \mathcal H, \ (\eta,\xi)=0$, then $\langle H(\eta\otimes \eta), \xi\otimes \xi\rangle = 0$.
\end{pro}
\begin{proof}
We can assume that  $\|\eta\|_{\mathcal H} = \|\xi\|_{\mathcal H}=1$. Since \ $p=\eta\otimes \eta$ is one dimensional projections and $H$ is a skew-Hermitian operator, it follows that
\begin{equation}\label{e13}
0=[H(p),p]= f_p(H(p))=\langle  H(p), p\rangle.
\end{equation}
By Lemma 9.2.7 (\cite[Ch. 9, \S 9.2]{flem2}, see also the proof of Lemma 11.3.2 \cite[Ch. 9, \S 11.3]{flem2}), there exists a  vector $\xi=\{\xi_1,\xi_2\}\in (\mathbb R^2,\|\cdot\|_E), \ \xi_1>0,\xi_2>0, \ \|\xi\|_E=1$, such that the  functional \ $f(\{\eta_1,\eta_2\})=\eta_1\xi_1+\eta_2\xi_2, \ \{\eta_1,\eta_2\}\in\mathbb R^2,$ is a  support functional at \ $\xi$ \ for space \ $(\mathbb R^2,\|\cdot\|_E)$.

Let us show  that the linear functional $$\varphi (y)=\langle y,x \rangle, \ y \in \mathcal C_E^h, \  x=\xi_1 p+\xi_2 q,$$
is a support functional at $x$ \ for \ $(\mathcal C_E^h,\|\cdot\|_E)$.

Since $f$ is support functional at $\xi$ \ for \ $(\mathbb R^2,\|\cdot\|_E)$ \ and \ $\|\xi\|_E=1$, it follows that $\xi_1^2+\xi_2^2 =f(\{\xi_1,\xi_2\})=f(\xi)= \|\xi\|_E^2=1.$
Furthermore, by  $\|f\| = \|\xi\|_E=1$, we have that \ $|f(\{\eta_1,\eta_2\})|=|\xi_1\eta_1+\xi_2\eta_2| \leq 1$ \ for every $\{\eta_1,\eta_2\}\in\mathbb R^2$  \ with $\|\{\eta_1,\eta_2\}\|_E \leq 1$.

Further, by  Lemma 4.1 \cite[Ch. II, \S 4]{gohb}, we have
$$|(y(\eta),\eta)|\leq \mu(1,y), \ |(y(\xi),\xi)|\leq \mu(1,y), \ |(y(\eta),\eta)|+ |(y(\xi),\xi)|\leq \mu(1,y)+\mu(2,y),
$$
that is, $\{(y(\eta),\eta),(y(\xi),\xi)\} \prec\prec \{\mu(1,y),\mu(2,y)\}$. Since $(E, \|\cdot\|_{E})$ is a fully symmetric sequence space, it follows that
$$\|\{(y(\eta),\eta),(y(\xi),\xi)\}\|_E \leq \|\{\mu(1,y),\mu(2,y)\}\|_E\leq \|y\|_{\mathcal C_E}.$$
Consequently, if  $y\in \mathcal C_E^h$ \ and \ $\|y\|_{\mathcal C_E} \leq 1$, then
$$|\varphi(y)|= |\langle y,x \rangle| = |\xi_1\mathrm{Tr}(py)+\xi_2\mathrm{Tr}(qy)|=
|f(\{(y(\eta),\eta),(y(\xi),\xi)\})|\leq 1,$$
that is, $\|\varphi\|_{(\mathcal C_E^h,\|\cdot\|_E)^*} \leq 1$.
Since $\|x\|_{\mathcal C_E}=\|\xi\|_E=1$ and
$$\varphi (x)=\langle x,x \rangle=\langle \xi_1 p+\xi_2 q, \xi_1 p+\xi_2 q \rangle=
\mathrm{Tr}(\xi_1 p+\xi_2 q)(\xi_1 p+\xi_2 q) )=\xi_1^2+\xi_2^2 =1,$$
it follows that $\|\varphi\|_{(\mathcal C_E^h,\|\cdot\|_E)^*} = 1 = \|x\|_{\mathcal C_E}$ \ and \ $\varphi (x)=\|x\|_{\mathcal C_E}^2$.
This means that $\varphi$ is a support functional at $x$ \ for space \ $(\mathcal C_E^h,\|\cdot\|_{\mathcal C_E})$.

Hence,
$$0=[H(x),x]=\varphi(H(x))=\langle H(x),x \rangle =\langle \xi_1 H(p)+\xi_2 H(q), \xi_1 p+\xi_2 q \rangle.$$
Since $\langle H(p), p\rangle=\langle H(q), q\rangle =0$ (see  (\ref{e13})), it follows that
\begin{equation}\label{e23}
\langle H(p),q \rangle=-\langle H(q),p \rangle.
\end{equation}

 We extend $\eta_1=\eta, \ \eta_2=\xi$  up to an orthonormal basis $\{\eta_i\}_{i=1}^{\infty}$, and let $p_i = \eta_i\otimes \eta_i$. Now we replace our operator $H$ with another skew-Hermitian operator $H_0$.
Let $u$ be a unitary operator such that  $u(\eta_1)=\eta_2, u(\eta_2)=\eta_1$ and $u(\eta_k)=\eta_k$ if $k\neq 1,2$. It is clear that $u^*=u^{-1}=u, \ up_1u=p_2, \  up_2u=p_1, \ up_iu=p_i, \ i\neq 1,2$, and $V(x)=uxu^*=uxu$ is an surjective isometry on $\mathcal C_E^h$, in addition,   $V^{-1} = V$.

By Proposition \ref{p31}, a linear operator $H_1 = VHV^{-1}$ is a skew-Hermitian operator, in particular, $\langle H_1(p_k), p_k \rangle =0$ for all $k\in\mathbb N$ (see  (\ref{e13})).

If $i,j\neq 1,2$,  then
\[
\langle H_1(p_i), p_j \rangle = \langle u H(p_i) u, p_j \rangle =\mathrm{Tr}(p_j u H(p_i) u)=(u H(p_i) u (\eta_j),\eta_j)=\]
\[ =(H(p_i) u (\eta_j), u^*(\eta_j))=(H(p_i)(\eta_j),\eta_j)=\mathrm{Tr}(p_j H (p_i))=\langle H(p_i),p_j \rangle.\]
If $i=1, \ j\neq 1,2$, then
\[\langle H_1(p_1), p_j \rangle = \langle uH(p_2)u, p_j \rangle =\mathrm{Tr}(p_j u H(p_2) u)=(u H(p_2) u(\eta_j),\eta_j)=\]
\[=(H(p_2)u(\eta_j), u^*(\eta_j))=(H(p_2)(\eta_j),\eta_j)=\mathrm{Tr}(p_jH(p_2))=\langle H(p_2),p_j \rangle.\]
Similarly, we get the following equalities

$(i)$. $\langle H_1(p_2), p_j \rangle=\langle H(p_1), p_j \rangle$ if  $i=2, \ j\neq 1,2$;

$(ii)$. $\langle H_1(p_i), p_1 \rangle =\langle H(p_i), p_2 \rangle$ if  $j=1, \ i\neq 1,2$;

$(iii)$. $\langle H_1(p_1), p_2 \rangle =\langle H(p_2), p_1 \rangle$ if  $i=1, \ j=2$;

$(iv)$. $\langle H_1(p_2), p_1 \rangle =\langle H(p_1), p_2 \rangle$ if   $i=2, \ j=1$.

It is clear that $H_0=\frac{1}{2} (H-H_1)$ is a skew-Hermitian operator, and if \ $i,j\neq 1,2$, \ then \ $\langle H_0(p_i), p_j \rangle = \frac{1}{2} (\langle H(p_i), p_j \rangle - \langle H_1(p_i), p_j \rangle) = 0$. Similarly, if $i=1, \ j\neq 1,2$ (respectively, $i=2, \ j\neq 1,2$) we get $\langle H_0(p_1), p_j \rangle = \frac{1}{2}(\langle H(p_1), p_j \rangle -\langle H(p_2), p_j \rangle)$ (respectively, $\langle H_0(p_2), p_j \rangle =\frac{1}{2}(\langle H(p_2), p_j \rangle -\langle H(p_1), p_j \rangle)$), that is, $\langle H_0(p_1), p_j \rangle+\langle H_0(p_2), p_j \rangle = 0$ in the case $j\neq 1,2$.

Similarly,  $\langle H_0(p_j), p_1 \rangle+\langle H_0(p_j), p_2 \rangle = 0$ if $j\neq 1,2$.
Since
$$\langle H_0(p_1), p_2 \rangle = \frac{1}{2}(\langle H(p_1), p_2 \rangle -\langle H(p_2), p_1 \rangle), \ \langle H(p_1), p_2 \rangle = -\langle H(p_2), p_1 \rangle \ (see (\ref{e23})),
$$
it follows that   $\langle H_0(p_1), p_2 \rangle = \langle H(p_1), p_2 \rangle$. Similarly, we get that $\langle H_0(p_2), p_1 \rangle = - \langle H(p_1), p_2 \rangle$. Finally, since $H_0$ is a skew-Hermitian operator, we have $\langle H_0(p_k), p_k \rangle =0$ for all $k\in\mathbb N$ (see (\ref{e13})).

Let $n$ be the smallest natural number such that the norm $\|\cdot\|_E$ is not Euclidian on $\mathbb R^n$. Then there  exist  (see,  \cite[Lemma 5.4]{ac}) linear independent vectors $\xi=(\xi_1,\xi_2,\ldots,\xi_n), \ \eta = (\eta_1,\eta_2,\ldots, \eta_n)\in\mathbb R^n, \ \|\xi\|_E = 1$, such that
\begin{equation}\label{e33}
\|\xi\|_E= \|f_\eta\|_{E^*}=f_\eta(\xi)=1,
\end{equation}
where $f_\eta(\zeta)=\sum\limits_{i=1}^n \zeta_i \eta_i, \ \zeta=(\zeta_1,\zeta_2,\ldots,\zeta_n) \in \mathbb R^n$. By rearranging the coordinates we may assume that $\xi_1\eta_2\neq \xi_2\eta_1$.

Let $x=\sum\limits_{j=1}^n \xi_j p_j$, \ $y=\sum\limits_{j=1}^n \eta_j p_j$,  and let $\varphi_y(z)=\langle z, y\rangle= \sum\limits_{j=1}^n \eta_j\cdot\mathrm{Tr}(p_j z), \ z \in \mathcal C_E^h$.

Let us show  that  $\varphi_y$ \ is a support functional at $x$ \ for \ $(\mathcal C_E^h,\|\cdot\|_E)$. \
Since $\|f_{\eta}\|_{E^*}=1$ (see (\ref{e33})), it follows that $|f_{\eta}(\zeta)|=|\sum\limits_{i=1}^n \eta_i\zeta_i| \leq 1$ for every $\zeta=\{\zeta_i\}_{i=1}^n \in \mathbb R^n$ with  \ $\|\zeta\|_E \leq 1$. Note that $\|x\|_{\mathcal C_E}=\|\xi\|_E=1$.

We should show that $\|\varphi_y\|=\|x\|_{\mathcal C_E}=1$ and $\varphi_y (x)=\|x\|_{\mathcal C_E}^2=1$. Indeed, $\varphi_y (x) = \langle x,y\rangle=\langle \sum\limits_{j=1}^n \xi_jp_j,\sum\limits_{j=1}^n \eta_jp_j\rangle=\sum\limits_{j=1}^n \xi_j\eta_j = f_{\eta}(\xi)=1=\|x\|_{\mathcal C_E}^2$.

If $z\in  \mathcal C_E^h, \ \|z\|_{ \mathcal C_E} \leq 1$ then $|\varphi_y(z)|= |\sum\limits_{j=1}^n \eta_j(z(\eta_j),\eta_j)| \leq 1.$
The last inequality follows from
$$\{(z(\eta_1),\eta_1),(z(\eta_2),\eta_2),\ldots ,(z(\eta_n),\eta_n)\}\prec \prec \{\mu(1,z),\mu(2,z),\ldots ,\mu(n,z)\}$$
(see  Lemma 4.1. \cite{gohb} Ch. II, \S 4). Therefore 
$$\|\varphi_y\|=\|x\|_{ \mathcal C_E}=1 \ \ \text{and} \ \ \varphi_y (x)=\|x\|_{ \mathcal C_E}^2=1.$$
This means that $\varphi_y$ is a support functional at $x$ \ for \ $( \mathcal C_E^h,\|\cdot\|_E)$.

Consequently,
\[0=\langle H_0(x),y\rangle=\langle \xi_1 H_0(p_1)+\ldots +\xi_n H_0(p_n), \eta_1 p_1+\ldots +\eta_n p_n \rangle=\]
\[=(\xi_1\eta_2-\xi_2\eta_1)\langle H_0(p_1), p_2 \rangle +(\xi_1\eta_3-\xi_2\eta_3)\langle H_0(p_1), p_3 \rangle+\ldots +\]
\[+ (\xi_1\eta_n-\xi_2\eta_n)\langle H_0(p_1), p_n \rangle+(\xi_3\eta_1-\xi_3\eta_2)\langle H_0(p_3), p_1 \rangle+\ldots+\]
\[ +(\xi_n\eta_1-\xi_n\eta_2)\langle H_0(p_n), p_1 \rangle. \ \ \ \ \ \  \eqno (4)\]

Let now $\tilde{x} = \xi_1 p_1+\xi_2 p_2-\xi_3 p_3-\ldots -\xi_n p_n$ and $\tilde{y} = \eta_1 p_1+\eta_2 p_2-\eta_3 p_3-\ldots -\eta_n p_n$. As above, we have that $\varphi_{\tilde{y}}(\cdot)=\langle\cdot, \tilde{y}\rangle$ is a support functional at $\tilde{x}$.  Consequently,
\[0=\langle H_0(\tilde{x}),\tilde{y}\rangle=(\xi_1\eta_2-\xi_2\eta_1)\langle H_0(p_1), p_2 \rangle +(-\xi_1\eta_3+\xi_2\eta_3)\langle H_0(p_1), p_3 \rangle+\ldots + \]
\[+(-\xi_1\eta_n+\xi_2\eta_n)\langle H_0(p_1), p_n \rangle+(-\xi_3\eta_1+\xi_3\eta_2)\langle H_0(p_3), p_1 \rangle+\ldots+\]
\[ +(-\xi_n\eta_1+\xi_n\eta_2)\langle H_0(p_n), p_1 \rangle. \ \ \ \ \ \  \eqno (5)\]
Summing $(4)$ and $(5)$ we obtain  $$2(\xi_1\eta_2-\xi_2\eta_1)\langle H_0(p_1), p_2 \rangle=0,$$
that is,  $$ \langle H(p_1), p_2 \rangle =\langle H_0(p_1), p_2 \rangle = 0.$$
\end{proof}

\begin{pro}\label{p33} Let $\eta\in \mathcal  H, \ \|\eta\|_{\mathcal H}=1, \ p=\eta\otimes \eta, \ x\in \mathcal K(\mathcal H)^h$, and let $\mathrm{Tr}(xq)=0$ for any one dimensional projection $q$ with $qp=0$ . Then there exists $f\in \mathcal H$ such that $x=\eta\otimes f+ f\otimes\eta - (\eta\otimes\eta)(f\otimes\eta), \ \|f\|_{\mathcal H}\leq \|x\|_\infty$.
\end{pro}
\begin{proof}
If $q$ is an one dimensional projection with $qp=0$ then $qxq = \alpha q$ for some $\alpha \in \mathbb R$, and $0=\mathrm{Tr}(xq)=\mathrm{Tr}(qxq) = \mathrm{Tr}(\alpha q)=\alpha,$
that is, $\alpha=0$ and $qxq =0$. Let
$$e\in \mathcal P(\mathcal H), \ \dim e(\mathcal H)=1, \ ep=0, \ eq=0, \ y=(q+e)x(q+e).$$
If $y\neq 0$ then there exists $r \in \mathcal P(\mathcal H), \ \dim r(\mathcal H)=1$ such that $r\leq q+e$ and $rxr = ryr=\beta r$ for some $0\neq \beta \in \mathbb R$. Since $rp=0$, it follows that $0=\mathrm{Tr}(xr)=\mathrm{Tr}(rxr) = \beta\neq 0$. Thus $y=0$. Continuing this process, we construct a sequence of finite-dimensional projections $g_n \uparrow (I-p)$ such that $g_nxg_n =0$ for all $n \in \mathbb N$, where $I(h)=h, \ h \in \mathcal H$. Consequently,  $(I-p)x(I-p)=0$.

If $f=x(\eta)$ then $xp = f\otimes\eta$ and $px = \eta\otimes f$. In addition,
 \[(I-p)xp(h)=  (I-p)x((h,\eta)\eta))=(h,\eta)(I-p)f , \ h \in \mathcal H, \]
that is, $(I-p)xp = (I-p)f\otimes\eta$. Therefore,
\[x = px+(I-p)xp =  \eta\otimes f+ (I-p)f\otimes\eta \ \ \text{and} \ \ \|f\|_{\mathcal H}\leq \|x\|_\infty.\]
\end{proof}

\begin{pro}\label{p34}
 Let $\eta\in \mathcal H, \ \|\eta\|_{\mathcal H}=1, \ p=\eta\otimes \eta$. Then there exists $f\in \mathcal H$ such that \[H(\eta\otimes \eta) = \eta\otimes f+ f\otimes \eta, \ \|f\|_{\mathcal H}\leq \|H\|.\]
\end{pro}
\begin{proof}
If $x=H(\eta\otimes \eta), \ \xi\in \mathcal H, \ (\eta,\xi)=0, \ q=\xi\otimes \xi,$ \ then by Proposition \ref{p32}  we obtain that $$(x(\xi),\xi)= \langle x, \xi\otimes\xi \rangle= \mathrm{Tr}(x\cdot \xi\otimes\xi) = 0.$$
Using Proposition \ref{p33}, we
have that there  exists $f\in \mathcal H$ such that
$$H(\eta\otimes \eta)=x=\eta\otimes f+ f\otimes \eta - (\eta\otimes\eta)(f\otimes\eta).$$
Since $H$ is a skew-Hermitian operator, it follows that
\[0= \langle H(\eta\otimes \eta),\eta\otimes \eta \rangle= \langle \eta\otimes f+ f\otimes \eta-(\eta\otimes \eta)(f\otimes \eta), \eta\otimes \eta\rangle =\]
\[= \mathrm{Tr}((\eta\otimes \eta)(\eta\otimes f+ f\otimes \eta-(\eta\otimes \eta)(f\otimes \eta)))
=\]
\[=\mathrm{Tr}((\eta\otimes \eta)(\eta\otimes f))=((\eta\otimes f)(\eta),\eta)=(\eta,f).\]
Thus $(\eta,f)=0$ and  $x=\eta\otimes f+ f\otimes \eta-(\eta\otimes \eta)(f\otimes \eta)=\eta\otimes f+ f\otimes \eta$.
In addition,
\[\|f\|_{\mathcal H}\leq \|x\|_\infty \leq \|x\|_{\mathcal C_E} =\|H(\eta\otimes \eta)\|_{\mathcal C_E}\leq \|H\|\cdot \|\eta\otimes \eta\|_{\mathcal C_E}=\|H\|\cdot \|\eta\otimes \eta\|_\infty = \|H\|.\]
\end{proof}

\begin{pro}\label{p35}
 There exists $a\in \mathcal B(\mathcal H)$ such that $H(x) = ax+xa^*$ for every $x\in \mathcal C_E^h$.
\end{pro}
\begin{proof}
Let $\{p_i\}_{i=1}^{\infty} = \{\eta_i \otimes \eta_i\}_{i=1}^{\infty}$ be a basis in  real linear  space  $ \mathcal F(\mathcal H)^h$, where $\{\eta_i\}_{i=1}^{\infty}$ is an orthonormal basis of $\mathcal H$. For every $\eta_i\in \mathcal H$ there exists $f_i\in \mathcal H$ such that $H(\eta_i \otimes \eta_i) = \eta_i\otimes f_i+ f_i\otimes \eta_i$, and $\|f_i\|_{\mathcal H}\leq \|H\|$ for all $i \in \mathbb N$ (see Proposition \ref{p34}). Define a linear operator $a\colon \mathcal H\to \mathcal H$  setting $a(\eta_i)=f_i$. Since $\|f_i\|_{\mathcal H}\leq \|H\|$ for all $i \in \mathbb N$, it follows that   $a\in \mathcal B(\mathcal H)$, in addition, $H(p_i) = \eta_i\otimes a(\eta_i)+ a(\eta_i)\otimes \eta_i$. Since $\eta_i\otimes a(\eta_i) = (\eta_i\otimes \eta_i)a^*$ and $a(\eta_i)\otimes \eta_i = a(\eta_i\otimes \eta_i)$, it follows that $H(x) = ax+xa^*$ for all $x\in \mathcal F(\mathcal H)^h$.

If  $(\mathcal C_E, \|\cdot\|_{\mathcal C_E})$ is a separable space then $\mathcal F(H)^h$ is dense in  $(\mathcal{C}_E^h, \|\cdot\|_{\mathcal C_E})$. Consequently, $H(x) = ax+xa^*$ for all $x \in \mathcal C_E^h$.

Let now   $(\mathcal C_E, \|\cdot\|_{\mathcal C_E})$ be a perfect Banach symmetric ideal. Repeating the proof of Theorem 4.4 \cite{ach}, we obtain that the skew-Hermitian operator $H$ is a  $\sigma(\mathcal{C}_E^h, \mathcal C_E^{\times})$-continuous. Since  the space $\mathcal F(\mathcal H)^h$ is  $\sigma(\mathcal{C}_E^h, \mathcal C_E^{\times})$-dense in  $\mathcal{C}_E^h$, it follows that $H(x) = ax+xa^*$ for all $x \in \mathcal C_E^h$.
\end{proof}
The following Proposition completes the proof of Theorem  \ref{t12}.
\begin{pro}\label{p36}
$a=ib$ for some $b\in \mathcal B(\mathcal H)^h$.
\end{pro}
\begin{proof} If $a=a_1+ia_2, \ a_1,a_2\in \mathcal B(\mathcal H)^h$, then
\[H(x)=ax+xa^* = a_1x+xa_1+i(a_2x-xa_2)=S_1(x_1)+S_2(x),\]
where $S_1(x) = a_1x+xa_1, \ S_2(x) = i(a_2x-xa_2), \ x \in \mathcal C_E^h$. Since $H$ and $S_2$ are skew-Hermitian, it follows that $S_1=H-S_2$ is also skew-Hermitian.

If $p \in \mathcal P(\mathcal H), \ \dim p(\mathcal H)=1$, then the lineal functional $\varphi(y) = \langle y,p\rangle = \mathrm{Tr}(yp), \ y \in \mathcal C_E^h$, is support functional at $p$. Thus  $\mathrm{Tr}(pa_1p+pa_1)=\mathrm{Tr}(S_1(p)p)=0,$
that is, $-\mathrm{Tr}(pa_1)=\mathrm{Tr}(pa_1p)=\mathrm{Tr}(pa_1).$
This means that $\mathrm{Tr}(pa_1)=0$ for all $p \in \mathcal P(\mathcal H)$ with $\dim p(\mathcal H)=1$. Consequently,  $\mathrm{Tr}(xa_1)=0$  for all $x\in \mathcal F(\mathcal H)$, and by   \cite[Lemma  2.1]{ak} we have $a_1=0$. Therefore,  $a=ia_2$.
\end{proof}

\section{Skew-Hermitian operators in Orlicz, Lorentz, and Marcinkiewicz ideals of compact operstors}

In this section we present applications of Theorem \ref{t12}  to Orlicz, Lorentz and Marcinkiewicz ideals of compact operators.

1. Let $\Phi$ be an {\it Orlicz function}, that is, $\Phi:[0,\ii)\to [0,\ii)$ is left-continuous, convex, increasing and such that $\Phi(0)=0$ and $\Phi(u)>0$ for some $u\ne 0$ (see, for example, \cite[Ch.\,2, \S\,2.1]{es}, \cite[Ch.\,4]{lt}).  Let
$$
 l_\Phi=\left \{\xi= \{ \xi_n \}_{n = 1}^\ii\in  \ell_{\infty}: \  \sum\limits_{n=1}^\ii \Phi\left (\frac {|\xi_n |}a \right )
<\ii \text { \ for some \ } a>0 \right \}
$$
be the corresponding {\it Orlicz sequence space},  and let
$$\| \xi\|_\Phi=\inf \left \{ a>0:\ \sum\limits_{n=1}^\ii \Phi\left (\frac {|\xi_n |}a \right ) \leq 1\right\}
$$
be the {\it Luxemburg norm} in $ l_\Phi$. It is well-known that  $( l_\Phi, \| \cdot\|_\Phi)$ is a symmetric sequence space and the norm $\| \cdot\|_\Phi$  has the Fatou property.

If $\Phi(u)>0$ for all $u\ne 0$, then $\sum\limits_{n=1}^\ii\Phi\left (\frac1a \right )=\ii$ for each $a>0$, hence $\mathbf 1=\{1,1,...\}\notin l_\Phi$ and $l_\Phi\su c_0$. Consequently, we can define Orlicz ideal of compact operators
$$
\mc C_\Phi=\mc C_{l_\Phi},\ \ \| x\|_\Phi=\| x\|_{\mc C_{l_\Phi}}, \ x \in \mc C_\Phi.
$$
Now, Theorem  \ref{t12}  yield the  following.

\begin{teo}\label{t41} Let $\Phi$ be an Orlicz function such that $\Phi(u)>0$ for all $u\ne 0$ and $l_\Phi \neq l_2$.  Let  $H\colon \mc C_\Phi^h \to \mc C_\Phi^h$  be a skew-Hermitian operator. Then there exists self-adjoint operator $a\in \mathcal B(\mathcal H)$ such that $H(x)=i(xa - ax)$ for all $x \in \mc C_\Phi^h$.
\end{teo}

2. Let $\psi$ be a concave function on $[0, \ii)$ with $\psi(0)=0$ and $\psi(t)>0$ for all $t>0$, and let
\[
\Lambda_\psi=\left \{\xi= \{ \xi_n \}_{n = 1}^{\ii} \in  \ell_{\infty}: \  \|\xi \|_{\psi} =
\sum\limits_{n=1}^\ii \xi^*_n(\psi(n)-\psi(n-1))<\ii\right \},
\]
the corresponding {\it Lorentz sequence space}. The pair $(\Lambda_\psi, \|\cdot\|_\psi)$ is a  symmetric sequence space and the norm $\| \cdot\|_{\psi}$  has the Fatou property (see, for example, \cite[Ch.\,II, \S\,5]{kps}). Besides, if $\psi(\ii)=\ii$, then $\mathbf 1\notin\Lambda_\psi$ and   $\Lambda_\psi\su c_0$.
Consequently, we can define Lorentz ideal of compact operators
\[
\mc C_\psi=\mc C_{\Lambda_\psi}, \ \ \| x\|_\psi=\| x\|_{\mc{C}_{\Lambda_\psi}}, \ x\in\mc C_\psi.
\]
Therefore, in the case $\varphi(\ii)=\ii$, Theorem  \ref{t12} is valid for any Lorentz ideal $\mc C_\psi$ of compact operators.

3. Let $\psi$ be as above, and let
\[
M_\psi=\left\{\xi=\{ \xi_n \}_{n = 1}^\ii\in \ell_{\infty}: \  \|\xi \|_{M_\psi} =
\sup\limits_{n\geq1} \frac1{\psi(n)}\sum\limits_{k=1}^n\xi^*_k<\ii\right\},
\]
the corresponding {\it Marcinkiewicz sequence space}. The space $(M_\psi, \|\cdot\|_{M_\psi})$ is a symmetric sequence space and the norm $\| \cdot\|_{M_\psi}$  has the Fatou property (see, for example, \cite[Ch.\,II, \S\,5]{kps}). In addition,
$\mathbf 1\notin M_\psi(\mathbb N)$ if and only if $\lim\limits_{t\to\ii}\frac{\psi(t)}{t}=0$ \cite[Ch.\,II, \S\,5]{kps}). Consequently, in this case, $M_\psi \subset c_0$ and we can define Marcinkiewicz  ideal of compact operators
\[
\mc C_{M_\psi}=\mc C_{M_\psi},\ \ \| x\|_{\mc C_{M_\psi}}=\| x\|_{\mc C_{M_\psi}}, \ x \in \mc C_{M_\psi}.
\]

Therefore, in the case $\lim\limits_{t\to\ii}\frac{\psi(t)}{t}=0$,   Theorem  \ref{t12} is valid for any Marcinkiewicz  ideal $\mc C_{M_\psi}$ of compact operators.

\end{document}